\documentclass[12pt]{elsarticle}
\usepackage[utf8]{inputenc}
\usepackage[english]{babel}
\usepackage{amsmath}
\usepackage{amssymb}
\usepackage{amsfonts}
\usepackage{mathtools}
\usepackage{dsfont}
\usepackage{rotating}
\usepackage{graphicx}
\usepackage{floatflt,epsfig}
\usepackage{lineno,hyperref}
\usepackage{enumerate}
\usepackage{colortbl}
\usepackage{array,tabularx,tabulary,booktabs}
\usepackage{longtable}
\usepackage{multirow}
\usepackage{wrapfig}
\usepackage{subcaption}
\usepackage[table]{xcolor}
\newcolumntype{^}{>{\currentrowstyle}}

\usepackage{amsthm}
\usepackage[left=1in,right=1in,top=1in,bottom=1in]{geometry}

\usepackage{amssymb,amsthm}

\def\N{{\mathbb N}}

\def\R{{\mathbb R}}

\def\F{{\mathbb F}}
\def\PG{\mbox{\rm PG}}

\newtheorem{lemma}{Lemma}[section]
\newtheorem{theorem}[lemma]{Theorem}
\newtheorem{corollary}[lemma]{Corollary}
\newtheorem{proposition}[lemma]{Proposition}
\newtheorem{remark}[lemma]{Remark}
\newtheorem{example}[lemma]{Example}

\journal{Arxiv}
\begin{document}
\renewcommand{\abstractname}{Abstract}
\renewcommand{\refname}{References}
\renewcommand{\tablename}{Table}
\renewcommand{\arraystretch}{0.9}
\thispagestyle{empty}
\sloppy

\begin{frontmatter}
\title{Divisible design graphs with selfloops}

\author[02]{Anwita Bhowmik}
\ead{bhowmikanwita@gmail.com}

\author[03]{Bart De Bruyn}
\ead{Bart.DeBruyn@UGent.be}

\author[01]{Sergey Goryainov}
\ead{sergey.goryainov3@gmail.com}

\address[01] {School of Mathematical Sciences, Hebei International Joint Research Center for Mathematics and Interdisciplinary Science, Hebei Key Laboratory of Computational Mathematics and Applications, Hebei Workstation for Foreign Academicians, Hebei Normal University, Shijiazhuang  050024, P.R. China}
\address[02] {Postdoctoral Research Station of Mathematics, Hebei Normal University, Shijiazhuang 050024, P.R. China}
\address[03] {Ghent University, Department of Mathematics, Computer Science and Statistics, Krijgslaan 281 (S9), 9000 Gent, Belgium}


\begin{abstract}
We develop a basic theory for divisible design graphs with possible selfloops (LDDG's), 
and describe two infinite families of such graphs, some members of which are also classical 
examples of divisible design graphs without loops (DDG's). Among the described 
theoretical results is a discussion of the spectrum, a classification of all examples 
satisfying certain parameter restrictions or having at most three eigenvalues, 
a discussion of the structure of the improper and the disconnected examples, 
and a procedure called dual Seidel switching which allows to construct new examples of LDDG's from others.
\end{abstract}

\begin{keyword}
divisible design graph; loop; projective space; polarity 
\vspace{\baselineskip}
\MSC[2010] 05C50\sep 05B05\sep 51E20
\end{keyword}
\end{frontmatter}

\section{Introduction} \label{sec1}

A $k$-regular graph $\Gamma$ on $v$ vertices is called a {\em divisible design graph} with parameters $(v,k,\lambda_1,\lambda_2,$ $m,n)$ if the vertex set $V$ of $\Gamma$ can be partitioned into $m$ classes of size $n$ such that any two distinct vertices of the same class have exactly $\lambda_1$ common neighbours and any two vertices of different classes have precisely $\lambda_2$ common neighbours. Divisible design graphs or DDG's for short were introduced in \cite{HKM} as a bridge between (group) divisible designs and graphs. In \cite{HKM}, the graphs under consideration were supposed to be undirected and without loops. A generalization of this concept to directed graphs was introduced in \cite{DK} and (further) studied under the name of {\em divisible design digraphs}. As far as the authors know, the concept has not been considered for graphs in which (self-)loops are allowed. Such a (self-)loop arises for every vertex that is adjacent to itself. It is common in the literature that a loop in a vertex $v$ contributes exactly $2$ to the degree $\deg(v)$ of $v$, while the neighbours distinct from $v$ still contribute 1 to this degree. For the treatment we give in the present paper, it will be necessary that we adopt the convention that loops also contribute exactly 1 to the degrees of the vertices.

Also the notion of complementary graph of a simple graph needs to be modified so that it naturally fits with the possible presence of loops. For a graph $\Gamma$ (with possible loops), its complement $\overline{\Gamma}$ will now be defined such that for every vertex $v$, the sum $\deg_\Gamma(x) + \deg_{\overline{\Gamma}}(x)$ of its degrees equals the total number of vertices. While the complementary graph $\Gamma$ of a $k$-regular simple graph $\Gamma$ on $v$ vertices is $(v-k-1)$-regular, for the graphs under consideration here, such a complementary graph will be $(v-k)$-regular. Following these conventions, we can then adopt the same definition as above for a divisible design graph in which loops are allowed. We will then talk about LDDG's instead of DDG's. The verification of the following claim is then obvious.

\begin{lemma} \label{lem1.1}
The complement of an LDDG with parameters $(v,k,\lambda_1,\lambda_2,m,n)$ is an LDDG with parameters $(v,v-k,v-2k+\lambda_1,v-2k+\lambda_2,m,n)$.
\end{lemma}

\medskip Similarly as in \cite{HKM}, we will call an LDDG {\em proper} if $m,n \geq 2$ and $\lambda_1 \not= \lambda_2$. The classes of the LDDG are then unambiguously defined, and form a partition of the vertex set, to which we will refer as the {\em canonical partition}. In Section \ref{sec3}, we develop the basic theory of LDDG's with possible loops. In particular, we classify all proper LDDG's with $\lambda_1=k$ or having at most three distinct eigenvalues, and discuss the structure of disconnected examples. Following the treatment as in \cite{HKM}, we show that there are at most five distinct eigenvalues and show how these eigenvalues and their multiplicities can be computed from $\lambda_1$, $\lambda_2$, $m$, $n$, the number $L$ of loops in the graph and the total number $N^\ast$ of ordered pairs $(u,v)$, where $u$ and $v$ are possibly equal adjacent vertices belonging to the same class of the canonical partition. We also prove various (in)equalities between the parameters $v$, $k$, $\lambda_1$, $\lambda_2$, $L$, $N^\ast$, and discuss a procedure called dual Seidel switching that can result in new LDDG's.

In the following two theorems, we construct two families of LDDG's, hereby referring the reader to Section \ref{sec2} for the basic concepts on projective geometries that are necessary for the understanding and the mathematical proofs of these results.

\begin{theorem} \label{theo1.2}
Consider in the projective space $\PG(m-1,q)$, $m \geq 3$, a polarity $\zeta$. Let $\pi$ be a subspace of dimension $k-1 \geq 0$ which is totally isotropic with respect to $\zeta$. Let $\Gamma$ be the graph defined on $V := \PG(m-1,q) \setminus \pi^\zeta$, by calling two (possibly equal) vertices $x$ and $y$ adjacent whenever $x \in y^\zeta$. Then $\Gamma$ is a proper LDDG whose canonical partition consists of the classes of the equivalence relation $R$ on $V$, defined by $(x,y) \in R$ if and only if $\langle x,\pi \rangle = \langle y,\pi \rangle$.  
\end{theorem}

\begin{theorem} \label{theo1.3}
Consider in the projective space $\PG(m-1,q)$, $m,q \geq 3$, a polarity $\zeta$ and suppose $x^\ast$ is a point of $\PG(m-1,q)$ not contained in the hyperplane $H^\ast := (x^\ast)^\zeta$. Let $\Gamma$ be the graph defined on $V := \PG(m-1,q) \setminus (H^\ast \cup \{ x^\ast \})$, by calling two (possibly equal) vertices $x$ and $y$ adjacent whenever $x \in y^\zeta$. Then $\Gamma$ is a proper LDDG whose canonical partition consists of the classes of the equivalence relation $R$ on $V$, defined by $(x,y) \in R$ if and only if $x^\ast x = x^\ast y$.
\end{theorem}

\medskip \noindent We will prove Theorems \ref{theo1.2} and \ref{theo1.3} in the respective Sections \ref{sec4} and \ref{sec5}. Note that the complementary graphs of the LDDG's defined in Theorems \ref{theo1.2} and \ref{theo1.3} are also LDDG's by Lemma \ref{lem1.1}. The constructed LDDG's and their complements usually have loops. However, there are also a few instances, where no loops at all occur, giving rise to classical examples of DDG's. One of those instances, arises by considering a symplectic polarity $\zeta$ and taking the complementary graph of the graph $\Gamma$ constructed in Theorem \ref{theo1.2}. This graph was first considered in \cite{GWL}, where it was shown to be a {\em Deza graph}, meaning that the number of common neighbours that two distinct vertices have only takes two values. The treatment given in Section \ref{sec3} can therefore be seen as a strengthening and generalization of this result. Additionally to what has been proved in \cite{GWL}, we also show here that the distribution of the two values of the numbers $|\Gamma(x) \cap \Gamma(y)|$ for distinct vertices $x$ and $y$ is compatible with a partition of the vertex set, necessarily being the canonical partition of a proper DDG. Other constructions of proper DDG's related to symplectic graphs can be found in the recent works \cite{Bh-Go,DB-Go-Ha-Sh,Ka}.   

\section{Basic notions on projective geometry} \label{sec2}

In this section, we review those basic notions and properties of finite projective geometries that are relevant for this paper. For proofs and more background information, we refer to the standard works \cite{Hir,Hi-Th}.  

For $m \in \N$, let $V$ be an $m$-dimensional vector space over the finite field $\F_q$. Associated with $V$, there is a projective space $\PG(V)=\PG(m-1,q)$, whose {\em points} consist of all $1$-dimensional subspaces of $V$. For every subspace $U$ of dimension $i+1$ of $V$, there is associated a so-called $i$-dimensional {\em projective subspace} $S_U$ of $\PG(V)$, which is the set consisting of all $1$-dimensional subspaces contained in $U$. We will then write $\dim(S_U)=i$. The projective subspaces of $\PG(V)$ will shortly be called the {\em subspaces} in the sequel. The projective space $\PG(V)$ itself is said to have dimension $m-1$. There are unique subspaces of dimension $-1$ and $m-1$, namely $S_{ \{ \bar o \}}=\emptyset$ and $S_V$. Every subspace of dimension 0 is a singleton consisting of a single point. The subspaces of dimension $1$, $2$ and $m-2$ are respectively called the {\em lines}, {\em planes} and {\em hyperplanes} of $\PG(V)$. Every two distinct points $x$ and $y$ of $\PG(V)$ are contained in a unique line which we will denote by $xy$.

The intersection of two subspaces $S_{U_1}$ and $S_{U_2}$ is again a subspace, namely the subspace $S_{U_1 \cap U_2}$. The smallest subspace containing $S_{U_1}$ and $S_{U_2}$ is equal to $S_{\langle U_1,U_2 \rangle}$ and is called the subspace of $\PG(V)$ {\em generated} by $S_{U_1}$ and $S_{U_2}$. We will denote this subspace also by $\langle S_{U_1},S_{U_2} \rangle$. The Grassmann identity for subspaces of a vector space immediately implies the {\em Grassmann identity for projective spaces}, namely that if $S_1,S_2$ are two projective subspaces, then $\dim(S_1)+\dim(S_2) = \dim(S_1 \cap S_2) + \dim(\langle S_1,S_2 \rangle)$. From this identity, it follows that if $H$ is a hyperplane and $\pi$ is an $i$-dimensional subspace not contained in $H$, then $\pi \cap H$ is an $(i-1)$-dimensional subspace. We will freely make use of these facts in the sequel.

Suppose $\pi$ is a subspace of $\PG(V)$ and $X$ is a set of points of a subspace $\pi'$ of $\PG(V)$ that is {\em complementary} to $\pi$ (meaning that $\pi \cap \pi' = \emptyset$ and $\langle \pi,\pi' \rangle = S_V$). Then the {\em cone} $\pi X$ with {\em kernel} $\pi$ and {\em basis} $X$ is the union of $\pi$ with all subspaces of the form $\langle \pi,x \rangle$, $x \in X$. Note that if $X \not= \emptyset$, then $\pi X = \bigcup_{x \in X} \langle \pi,x \rangle$. 

If $\kappa$ is a quadratic form on $V$, then the set of all points $\langle \bar x \rangle$ of $\PG(V)$ for which $\kappa(\bar x)=0$ is a {\em quadric} of $\PG(V)$, and such a quadric is called {\em nonsingular} if through each point $p \in Q$ there is at least one line intersecting $Q$ in exactly two points. If $w-1$ is the maximal projective dimension of a subspace contained in a given nonsingular quadric $Q$, then $w$ is called the {\em Witt index} of $Q$. There are three types of nonsingular quadrics in $\PG(V)=\PG(m-1,q)$. If $m$ is odd, then the quadric is a so-called {\em parabolic quadric} having Witt index $\frac{m-1}{2}$, which is usually denoted by $Q(m-1,q)$. The parabolic quadric $Q(0,q)$ of $\PG(0,q)$ is empty, while $Q(2,q)$ is also called an {\em irreducible conic} of $\PG(2,q)$. If $m$ is even, then the quadric has either Witt index $\frac{m}{2}$, in which case the quadric is called a {\em hyperbolic quadric} and denoted by $Q^+(m-1,q)$ or Witt index $\frac{m-2}{2}$ (with $m \geq 2$) in which case the quadric is called an {\em elliptic quadric} and denoted by $Q^-(m-1,q)$. In $\PG(1,q)$ and elliptic quadric $Q(1,q)$ is empty, while the hyperbolic quadric $Q^+(1,q)$ contains exactly two points. By convention, we put $Q^+(-1,q)$ equal to the empty point set, and leave $Q^-(-1,q)$ undefined. Putting $m$ equal to $2n+1$ or $2n$ depending on whether it is odd or even, and $n \geq 1$ for the elliptic case, these quadrics have the following sizes:
\[ |Q(2n,q)| = \frac{q^{2n}-1}{q-1},\  |Q^+(2n-1,q)| = \frac{q^{2n-1}-1}{q-1}+q^{n-1},\  |Q^-(2n-1,q)| = \frac{q^{2n-1}-1}{q-1}-q^{n-1}. \]

Suppose $q=r^2$ is a square and $b$ is a Hermitean form on $V$, implying that $b(\bar w,\bar v)=b(\bar v,\bar w)^r$ and $b(\lambda \bar v,\mu \bar w)=\lambda \mu^r \cdot b(\bar v,\bar w)$ for all $\bar v,\bar w \in V$ and all $\lambda,\mu \in \F_q$. Then the set $\mathcal{H}$ of all points $\langle \bar x \rangle$ of $\PG(V)$ for which $b(\bar x,\bar x)=0$ is a Hermitean variety of $\PG(V)$, and such a Hermitean variety $\mathcal{H}$ is called {\em nonsingular} if each point $x \in \mathcal{H}$ is contained in a line intersecting $\mathcal{H}$ is exactly $r+1$ points. Up to projective equivalence, there is only one type of nonsingular Hermitean variety in $\PG(m-1,q)$. This Hermitean variety is usually denoted by $H(m-1,q)$ and its Witt index (being defined similarly) is equal to either $\frac{m}{2}$ or $\frac{m-1}{2}$ depending on whether $m$ is even or odd. We have
\[ |H(m-1,q)| = \frac{(r^m+(-1)^{m+1})(r^{m-1}+(-1)^m)}{q-1}.  \]
The nonsingular Hermitean varieties $H(-1,q)$ and $H(0,q)$ in respectively $\PG(-1,q)$ and $\PG(0,q)$ are empty.
 
Suppose $m \geq 3$ and denote by $\Sigma$ the set of all subspaces of $\PG(V)$. An {\em automorphism} of $\PG(V)$ is a permutation $\theta$ of $\Sigma$, respecting the inclusion of subspaces, meaning that if $U_1,U_2 \in \Sigma$, then $U_1 \subseteq U_2$ if and only if $U_1^\theta \subseteq U_2^\theta$. An {\em anti-automorphism} of $\PG(V)$ is a permutation $\theta$ of $\Sigma$ reversing the inclusion relation, meaning that if $U_1,U_2 \in \Sigma$, then $U_1 \subseteq U_2$ if and only if $U_2^\theta \subseteq U_1^\theta$. An anti-automorphism $\theta$ for which $\theta^2$ is the identity is called a {\em polarity}. If $\zeta$ is a polarity and $\pi_1,\pi_2$ are two subspaces such that $\pi_1 \subseteq \pi_2^\zeta$, then we must have $\pi_2 \subseteq \pi_1^\zeta$. A subspace $\pi$ is called {\em totally isotropic} with respect to a polarity $\zeta$ if $\pi \subseteq \pi^\zeta$.

There are four types of polarities in a finite projective spaces. Each of these polarities is obtained in a similar way, namely by considering some nondegenerate bilinear or Hermitean form $b$ on $V$. For every subspace $U$ of $V$, we denote by $U^\perp$ the subspace consisting of all vectors $\bar v$ for which $b(\bar u,\bar v)=0$ for all $\bar u \in U$. Then the map $S_{U} \mapsto S_{U^\perp}$ is a polarity $\zeta_b$.

\medskip \noindent (1) If $b$ is a nondegenerate alternating bilinear form, then $m$ must be even and $b(\bar v,\bar v)=0$, $b(\bar v,\bar w)=-b(\bar w,\bar v)$ for $\bar v,\bar w \in V$. In this case, $\zeta_b$ is called a {\em symplectic polarity}. If $\zeta$ is a symplectic polarity, then $x \in x^\zeta$ for every point $x$ of $\PG(V)$.

\medskip \noindent (2) Suppose $q$ is even and $b$ is a nondegenerate symmetric bilinear form such that there is at least one $\bar v$ for which $b(\bar v,\bar v) \not= 0$ (without that latter condition, we would be in case (1) again). In this case, $\zeta_b$ is called a {\em pseudo-polarity}. If $\zeta$ is a pseudo-polarity, then in the finite case (as we consider here) the set of all points $x$ for which $x \in x^\zeta$ is a hyperplane $H^\ast$. The point $(H^\ast)^\zeta$ is then contained in the hyperplane $H^\ast$ if and only if $n$ is odd. 
 
\medskip \noindent (3)  Suppose $q$ is odd and $b$ is a nondegenerate symmetric bilinear form. In this case, $\zeta_b$ is called an {\em orthogonal polarity}. If $\zeta$ is an orthogonal polarity, then the set of all points $x$ for which $x \in x^\zeta$ is a nonsingular quadric $Q$ of $\PG(V)$ (which is of one of the three mentioned types). If $\pi$ is a subspace contained in $Q$, then $\pi^\zeta$ intersects $Q$ in a cone of type $\pi Q'$, where $Q'$ is a nonsingular quadric of the same type of $Q$ (parabolic, hyperbolic or elliptic) in a subspace of $\pi^\zeta$ that is complementary to $\pi$.

\medskip \noindent (4) Suppose $q=r^2$ is a square and $b$ is a nondegenerate Hermitean form. In this case, $\zeta_b$ is called a {\em unitary} or {\em Hermitean polarity}. If $\zeta$ is a unitary polarity, then the set of all points $x$ of $\PG(V)$ for which $x \in x^\zeta$ is a {\em nonsingular Hermitean variety} $\mathcal{H}$. If $\pi$ is a subspace contained in $\mathcal{H}$, then $\pi^\zeta$ intersects $\mathcal{H}$ in a cone of type $\pi \mathcal{H}'$, where $\mathcal{H}'$ is a nonsingular Hermitean variety in a subspace of $\pi^\zeta$ that is complementary to $\pi$.
 
 \section{Divisible design graphs with possible loops} \label{sec3}
 
DDG's were introduced in the paper \cite{HKM}, where some of their basic properties were derived. We give here a summary of some of the results from \cite{HKM} that extend to LDDG's, and derive several additional new results. Some of these results make use of the following notational convention: if $G$ is a loopless graph, then $\widetilde{G}$ denotes the graph obtained from $G$ by adding a loop to all of its vertices. Starting from the complete graph $K_l$, $l \geq 1$, and the complete bipartite graph $K_{l_1,l_2}$, $l_1,l_2 \geq 1$, we find in this way the graphs $\widetilde{K_l}$ and $\widetilde{K_{l_1,l_2}}$. Note also that the spectrum of $\widetilde{G}$ is obtained by adding 1 to the eigenvalues of the spectrum of $G$. 

\subsection{Basic (in)equalities and classification results} \label{sec3.1}

Recall that by definition, the following conditions certainly hold in a proper LDDG $\Gamma$ with parameters $(v,k,\lambda_1,\lambda_2,m,n)$:
\[ m \geq 2, \qquad n \geq 2, \qquad v=mn \geq 4, \qquad \lambda_1 \not= \lambda_2.  \]
The following proposition shows that also $k$ can be expressed in terms of $\lambda_1$, $\lambda_2$, $m$ and $n$.

\begin{proposition} \label{prop3.1}
We have $k(k-1) = (n-1) \lambda_1 + (m-1)n \lambda_2$. 
\end{proposition}
\begin{proof}
By double counting the ordered triples $(a,b,c)$ such that $a$ and $b$ are two distinct vertices and $c$ is a common neighbour of $a$ and $b$ (possibly equal to them), yields $v \cdot (n-1) \cdot \lambda_1 + v \cdot (m-1)n \cdot \lambda_2 = v \cdot k \cdot (k-1)$, from which the claim follows.
\end{proof}

\medskip \noindent The following is an immediate consequence of Lemma \ref{lem1.1}.

\begin{proposition} \label{prop3.2}
The complement of a proper LDDG is a proper LDDG.
\end{proposition}

\begin{proposition} \label{prop3.3}
We have $2 \leq k \leq v-2$.
\end{proposition}
\begin{proof}
In view of Proposition \ref{prop3.2}, we need to prove that $k \not= 0$ and $k \not= 1$. If $k=0$, then $\lambda_1=\lambda_2=0$. If $k=1$, then $\Gamma$ is a disjoint union of loops and $K_1$'s, and again we have $\lambda_1=\lambda_2=0$. This is in contradiction with the condition $\lambda_1 \not= \lambda_2$ for a proper LDDG.
\end{proof}

\medskip \noindent We obviously have $0 \leq \lambda_1 \leq k$. We now construct an infinite family of proper LDDG's for which $k=\lambda_1$. Let $\mathcal{D}$ be a symmetric design \cite{Be-Ju-Le} on the vertex set $\{ 1,2,\ldots,m \}$, $m \geq 3$, all whose $m$ blocks have size $l_1$ such that every two distinct blocks intersect in exactly $l_2$ points, with $0 < l_2 < l_1 < m$. Let $X_1,X_2,\ldots,X_m$ be mutually disjoint sets of size $n \geq 2$ and let $\theta$ be a bijection between $\{ 1,2,\ldots,m \}$ and the set of blocks such that if $i,j \in \{ 1,2,\ldots,m \}$, then $i \in j^\theta$ if and only if $j \in i^\theta$. Then consider the graph on the vertex set $X_1 \cup X_2 \cup \cdots \cup X_m$, where two possibly equal vertices $u$ and $v$ are adjacent whenever $j \in i^\theta$ if $i,j \in \{ 1,2,\ldots,m \}$ such that $u \in X_i$ and $v \in X_j$. Then obviously $\Gamma$ is a proper LDDG with parameters $(mn,l_1n,l_1n,l_2n,m,n)$ and canonical partition $\{ X_1,X_2,\ldots,X_m \}$. The number of loops in $\Gamma$ is $n$ times the number of $i$ for which $i \in i^\theta$.

Another trivial infinite family of such LDDG's can be obtained as a disjoint union of $m_1 \geq 0$ isomorphic copies of $\widetilde{K_n}$ and $m_2 \geq 0$ isomorphic copies of $K_{n,n}$, with $n \geq 2$ and $m = m_1 + 2m_2 \geq 2$. Then $\Gamma$ is a proper LDDG with parameters $(mn,n,n,0,m,n)$ whose canonical partition consists of all the $\widetilde{K_n}$'s and all maximal cocliques of the $K_{n,n}$'s.

\begin{proposition} \label{prop3.4}
We have $0 \leq \lambda_1 \leq k$. If $\Gamma$ is a proper LDDG with parameter $\lambda_1$ equal to $k$, then $\Gamma$ can be obtained from a symmetric design in the above-described way, or is a disjoint union of $\widetilde{K_n}$'s and $K_{n,n}$'s. 
\end{proposition}
\begin{proof}
Let $\{ X_1,X_2,\ldots,X_m \}$ denote the canonical partition of the LDDG. As $\lambda_1=k$, we know that if $u_1,u_2,v$ are vertices such that $u_1,u_2$ belong to the same class, then $u_1 \sim v$ if and only if $u_2 \sim v$. This implies that there is a map $\theta$ between $\{ 1,2,\ldots,m \}$ and a set of subsets of $\{ 1,2,\ldots,m \}$ such that for vertices $u$ and $v$, it holds that $u \sim v$ if and only if $j \in i^\theta$ if $i,j \in \{ 1,2,\ldots,m \}$ such that $u \in X_i$ and $v \in X_j$. Note that as adjacency is a symmetric relation, we have $j \in i^\theta$ if and only if $i \in j^\theta$ for $i,j \in \{ 1,2,\ldots,m \}$. As $\Gamma$ is $k$-regular, we have $|i^\theta| = l_1 := \frac{k}{n}$ for every $i \in \{ 1,2,\ldots,m \}$. As $\lambda_1 \not= \lambda_2$, it is impossible that every vertex of $\Gamma$ is adjacent to every other vertex, and we have that $l_1 < m$.

Now, take two distinct $i,j \in \{ 1,2,\ldots,m \}$ and let $u \in X_i$ and $v \in X_j$. As $u$ and $v$ have exactly $\lambda_2$ common neighbours, we have that $|i^\theta \cap j^\theta| = l_2 := \frac{\lambda_2}{n}$. Obviously, we have $\lambda_2 \leq \lambda_1$. If $\lambda_2 = \lambda_1$, then all $i^\theta$'s are equal to the same set. If $j \in \{ 1,2,\ldots,m \}$ does not belong to this set, then no vertex of $X_j$ can be adjacent to another vertex, implying that $k=0$ in contradiction with Proposition \ref{prop3.3}. So, $\lambda_2 < \lambda_1$ and $l_2 < l_1$. We thus have a collection of $m$ mutually distinct subsets of size $l_1 < m$ which two by two intersect in $l_2 < l_1$ elements. If $l_2 > 0$, this gives rise to a symmetric design. If $l_2=0$, then the $i^\theta$'s are mutually disjoint, partitioning the set $\{ 1,2,\ldots,m \}$, implying that $l_1=1$. In this case, we get a disjoint union of $\widetilde{K_n}$'s and $K_{n,n}$'s. 
\end{proof}

\begin{proposition} \label{prop3.5}
We have $0 \leq \lambda_2 \leq k-1$.
\end{proposition}
\begin{proof}
We obviously have $0 \leq \lambda_2 \leq k$. We need to show that the case $\lambda_2=k$ cannot occur. If $\lambda_2=k$, then two vertices belonging to different classes of the canonical partition have the same set of neighbours. Applying this observation two times, we see that also every two vertices of the same class have the same set of neighbours. A vertex in this common set of neighbours is then adjacent to $v$ vertices, in contradiction with Proposition \ref{prop3.3}.
\end{proof}

\medskip \noindent The upper bound in Proposition \ref{prop3.5} will be strengthened in Proposition \ref{prop3.18}.

\begin{proposition} \label{prop3.6}
\begin{enumerate}
\item[$(1)$] Let $\Gamma$ be a $k$-regular connected graph (with possible loops). Then $k$ is an eigenvector of $\Gamma$ with multiplicity 1 and corresponding eigenvector $[1,1,\ldots,1]^t$.
\item[$(2)$] Let $\Gamma$ be a connected graph with possible loops having $s$ mutually distinct eigenvalues. Then the diameter of $\Gamma$ is less than $s$. 
\end{enumerate}
\end{proposition}
\begin{proof}
These are classical results for graphs without loops (see e.g. \cite{Br-Ha}), whose proofs straightforwardly extend to the case of graphs with loops.
\end{proof}

\begin{proposition} \label{prop3.7}
\begin{enumerate}
\item[$(1)$] Let $\Gamma$ be a connected $k$-regular graph with exactly one eigenvalue. Then $\Gamma$ consists of a single vertex $x$. If $x$ has no loop, then the eigenvalue equals $k=0$, otherwise the eigenvalue equals $k=1$.
\item[$(2)$] Let $\Gamma$ be a connected $k$-regular graph with exactly two eigenvalues. Then $\Gamma$ is either a complete graph $K_l$ for some $l \geq 2$ having eigenvalues $k=l-1$ and $-1$ or $\widetilde{K_l}$ for some $l \geq 2$ having eigenvalues $k=l$ and $0$.
\end{enumerate}
\end{proposition}
\begin{proof}
(1) By Proposition \ref{prop3.6}(2), the diameter of $\Gamma$ is equal to $0$, from which the claim follows. 

(2) In this case, the diameter is at most and hence equal to $1$. So, $\Gamma$ can be obtained by adding a number of loops to a complete graph $K_l$ with $l \geq 2$. As the graph is regular, there are either no loops, or every vertex has a loop.
\end{proof}

\begin{corollary} \label{co3.8}
Let $\Gamma$ be a regular graph with exactly one eigenvalue. Then either $\Gamma$ is a graph without any edges or loops, or $\Gamma$ is a graph with a loop in every vertex, but with no additional edges.
\end{corollary}
\begin{proof}
All connected components then have the same unique eigenvalue. The claim then follows from Proposition \ref{prop3.7}(1).
\end{proof}

\begin{corollary} \label{co3.9}
Let $\Gamma$ be a regular graph with exactly two eigenvalues. Then one of the following cases occurs.
\begin{itemize}
\item[$(a)$] $\Gamma$ is a graph without proper edges, in which some vertices have loops, and others not. In this case, the eigenvalues are $0$ and $1$.
\item[$(b)$] $\Gamma$ is the disjoint union of a number of isomorphic copies of the same $\widetilde{K_l}$, $l \geq 2$, and some separate vertices without loops, with both being present. In this case, the eigenvalues are $0$ and $l$.
\item[$(c)$] $\Gamma$ is the disjoint union of a number of isomorphic copies of $K_2$ and some separate vertices with loops, with both being present. In this case, the eigenvalues are $-1$ and $1$. 
\item[$(d)$] $\Gamma$ is the disjoint union of a number of isomorphic copies of the same $K_l$, $l \geq 2$. In this case, the eigenvalues are $-1$ and $l-1$. 
\item[$(e)$] $\Gamma$ is the disjoint union of a number of isomorphic copies of the same $\widetilde{K_l}$, $l \geq 2$. In this case, the eigenvalues are $0$ and $l$.
\end{itemize}
\end{corollary}
\begin{proof}
By Proposition \ref{prop3.6}(2), all connected components of $\Gamma$ have diameter 0 or 1. If all of them have diameter 0, then we must have case (a). If there is a connected component of diameter 0 consisting of a vertex without loop and a connected component of diameter 1, then we must  have case (b). If there is a connected component of diameter 0 consisting of a vertex with loop and a connected component of diameter 1, then we must  have case (c). If all connected components have diameter 1, then either case (d) or case (e) occurs.
\end{proof}

\medskip \noindent The following is a consequence of Proposition \ref{prop3.3} and Corollaries \ref{co3.8} and \ref{co3.9}.

\begin{corollary} \label{co3.10}
If $\Gamma$ is a proper LDDG with at most two eigenvalues, then $\Gamma$ is either the disjoint union of at least two isomorphic copies of the same $K_l$ with $l \geq 3$, or the disjoint union of at least two isomorphic copies of the same $\widetilde{K_l}$ with $l \geq 2$.
\end{corollary}

\medskip \noindent A {\em strongly regular graph} $\Gamma$ with parameters $(v,k,\lambda,\mu)$ is defined as a $k$-regular graph of diameter 2 on $v$ vertices without loops such that every two adjacent vertices have exactly $\lambda$ common neighbours and any two distinct non-adjacent vertices have exactly $\mu$ common neighbours. If we add loops to every vertex of $\Gamma$, then we get a {\em modified strongly regular graph}. This graph $\widetilde{\Gamma}$ is then regular of degree $k+1$, any two distinct adjacent vertices have exactly $\lambda+2$ common neighbours and any two distinct nonadjacent vertices still have precisely $\mu$ common neighbours.  

\begin{proposition} \label{prop3.11}
Let $\Gamma$ be a connected $k$-regular graph with exactly three eigenvalues. Then $\Gamma$ is either a strongly regular graph or a modified strongly regular graph.
\end{proposition}
\begin{proof}
Denote by $k$, $l_1$ and $l_2$ the three eigenvalues of $\Gamma$, and $A$ the adjacency matrix (recall Proposition \ref{prop3.6}(2)). Put $B:= (A - l_1 I) (A - l_2 I)$. Then $B \not= O$ as the minimal polynomial of $A$ has degree 3. As $(A - k I) B = O$, we have $AB = kB$, and so every column of $B$ is a multiple of $[1 \, 1 \, \ldots \,1]^t$ by Proposition \ref{prop3.6}(1). As $B$ is symmetric, it is a nonzero multiple of $J$. This implies that we can write $A^2$ as a linear combination of $I$, $A$ and $J$, i.e. 
\[ A^2 = \alpha_1 I + \alpha_2 A + \alpha_3(J-I-A)  \]
for certain real numbers $\alpha_1$, $\alpha_2$ and $\alpha_3$. Comparing the entries at both sides of the inequality, we find:

$\bullet$ any two distinct nonadjacent vertices have exactly $\alpha_3$ common neighbours;

$\bullet$ any two distinct adjacent vertices have exactly $\alpha_2$ common neighbours; 

$\bullet$ every vertex without loop is adjacent to $\alpha_1$ vertices;

$\bullet$ every vertex with loop is adjacent to $\alpha_1 + \alpha_2$ vertices. 

\noindent Suppose $\alpha_2 = 0$. Then there cannot be loops in $\Gamma$ as for any edge containing a vertex with a loop, this vertex is adjacent to the two end vertices of the edge. In this case, $\Gamma$ is a strongly regular graph with parameters $(v,k,\lambda,\mu)=(v,\alpha_1,0,\alpha_3)$.

Suppose $\alpha_2 \not= 0$. Then it follows from the above conditions that either none or all vertices of $\Gamma$ have loops, resulting in either a strongly regular graph with parameters $(v,k,\lambda,\mu)=(v,\alpha_1,\alpha_2,\alpha_3)$ or a modified strongly regular graph for which the underlying strongly regular graph has parameters  $(v,k,\lambda,\mu)=(v,\alpha_1+\alpha_2-1,\alpha_2-2,\alpha_3)$.

Note that in both cases, the diameter has to be at least 2, as otherwise the graph would be isomorphic to either $K_v$ or $\widetilde{K_v}$ by Proposition \ref{prop3.7}, contrary to the fact that there are exactly three eigenvalues. Note also that in all cases, the numbers $\alpha_1$, $\alpha_2$ and $\alpha_3$ are nonnegative integers. 
\end{proof}

\begin{proposition} \label{prop3.12}
Suppose $\Gamma$ is a connected proper LDDG with exactly three eigenvalues. Then $\Gamma$ is isomorphic to either a complete multipartite graph of type $K_{n,n,\ldots,n}$ or of type $\widetilde{K_{n,n,\ldots,n}}$.
\end{proposition}
\begin{proof}
By Proposition \ref{prop3.11}, there exists a strongly regular graph $G$ such that $\Gamma$ is either $G$ or $\widetilde{G}$. Every two distinct adjacent vertices of $\Gamma$ have a constant number of common neighbours, as well as any two distinct nonadjacent vertices of $\Gamma$. So, the equivalence relation defining the classes of the LDDG are defined by either the adjacency relation or by the non-adjacency relation of $G$. As $G$ is connected, it must be the non-adjacency relation that is involved, and it follows that $G$ is a multipartite graph (of the mentioned type). 
\end{proof}

\begin{proposition} \label{prop3.13}
Let $\Gamma$ be a disconnected proper LDDG with exactly three eigenvalues. Then one of the following cases occurs for $\Gamma$:
\begin{enumerate}
\item[$(1)$] $\Gamma$ is the disjoint union of a nonzero number of $K_{n,n}$'s and a number of $\widetilde{K_n}$'s, for a certain integer $n \geq 2$;  
\item[$(2)$] there exist parameters $k$ and $l$, such that $\Gamma$ is the disjoint union of at least two graphs of one of the following types:
\begin{itemize}
\item a strongly regular graph with valency $k$ and with both the $\lambda$- and $\mu$-parameters equal to $l$; 
\item a modified strongly regular graph for which the underlying strongly regular graph has valency $k-1$, $\lambda$-parameter equal to $l-2$ and $\mu$-parameter equal to $l$.
\end{itemize}
\end{enumerate}
\end{proposition}
\begin{proof}
Denote by $k$ the valency of $\Gamma$. By Proposition \ref{prop3.3}, we know that $k \geq 2$. If we consider two vertices in distinct connected components along with their common neighbours (of which there are none), we see that either $\lambda_1=0$ and $\lambda_2=0$. Suppose that distinct vertices that have no common neighbours belong to the same class of the canonical partition. Then vertices belonging to distinct components belong to the same class. Applying this observation two times, we then see that also vertices belonging to the same connected component belong to the same class. This is in contradiction with the fact that there are at least two classes. So, $\lambda_1 > 0$ and $\lambda_2=0$.   

Connected components having exactly one eigenvalue cannot exist. Indeed, by Proposition \ref{prop3.7}(1), such a connected component is a vertex (with or without a loop) and so the valency $k$ must be either 0 or 1, a contradiction.

Suppose that all connected component have exactly two eigenvalues. Then by Proposition \ref{prop3.7}(2) and the fact that the valency $k$ is at least $2$, we know that every induced subgraph on each such connected component is either a graph of type $K_{k+1}$ or a graph of type $\widetilde{K_k}$. But then not all classes have the same size. So, this situation cannot occur. 

Suppose $\Gamma$ contains a connected component $C$ with exactly three eigenvalues. By Proposition \ref{prop3.11}, there is a strongly regular graph (having exactly three eigenvalues) such that $\Gamma$ is either $G$ or $\widetilde{G}$. Assume the parameters of $G$ as a strongly regular graph are $(|C|,k',\lambda,\mu)$. As $\mu > 0$, we have $\lambda_1=\mu$. If $\Gamma = G$, then $k'=k$ and if $\Gamma = \widetilde{G}$, then $k' = k-1$. Let us now distinguish between four cases.
\begin{enumerate}
\item[(a)] Suppose $\Gamma = G$ and $\lambda \not= \mu$. Then $\lambda = \lambda_2=0$, The relation of being nonadjacent defines an equivalence relation on the vertex set of $\Gamma$, and as $\lambda=0$, $\Gamma$ must be a complete bipartite graph of type $K_{k,k}$. The eigenvalues are then $k$, $0$ and $-k$.
\item[(b)] Suppose $\Gamma = G$ and $\lambda = \mu$. As the diameter of $G$ is 2, we have $k-\mu-1=k-\lambda-1 > 0$ and hence $k-\mu \geq 2$. Using the known formulas for the eigenvalues of a strongly regular graph in terms of the parameters (see e.g. \cite{Br-Ha}), we find that the eigenvalues of $\Gamma$ are equal to $k$, $\sqrt{k-\mu} > 1$ and $-\sqrt{k-\mu} < -1$.
\item[(c)] Suppose $\Gamma = \widetilde{G}$ and $\lambda = \mu-2$. As the diameter of $G$ is equal to $2$, we find that $(k-1)-\lambda-1 = k - \mu > 0$ and so $k-\mu \geq 1$. We can even prove that $k-\mu \geq 2$. Indeed, suppose that $k-\mu=1$. Then $\mu=k-1$ and $\lambda = k-3$, with $k \geq 3$. The number of vertices in $G$ at distance $2$ from a given vertex $x$ is equal to $k(k-\lambda-1)\mu^{-1} = \frac{2k}{k-1}$. This number is only integral for $k=3$, in which case $\mu=2$ and $\lambda=0$. Now, each of the three vertices of $G$ at distance 2 from $x$ is adjacent to exactly 2 vertices at distance 1 from $x$, and as $k=3$, this implies that these three vertices form a triangle, in contradiction with the fact that $\lambda=0$. So, we indeed have that $k - \mu \geq 2$. Computing the eigenvalues of $G$ with the known formulas, we find that these are equal to $k-1$, $-1+\sqrt{k-\mu}$ and $-1-\sqrt{k-u}$. So, the eigenvalues of $\Gamma = \widetilde{G}$ are again equal to $k$, $\sqrt{k-\mu} > 1$ and $-\sqrt{k-\mu} < -1$.
\item[(d)] Suppose $\Gamma = \widetilde{G}$ and $\lambda \not= \mu-2$. In this case, we have $\lambda_2 = \lambda+2$, in contradiction with the fact that $\lambda_2=0$. So, this case cannot occur. 
\end{enumerate}
We thus see that case (a) can never coexist with case (b) or case (c). 

Suppose case (a) occurs. If there is also a connected component with exactly two eigenvalues, then as the valency is equal to $k$ and all eigenvalues are $k$, $0$ and $-k$, we see by Proposition \ref{prop3.7}(2) that the graph induced on such a connected component is isomorphic to $\widetilde{K_k}$. We have case (1) of the proposition. 

Suppose case (b) and/or case (c) occurs. Then we show that we have case (2) of the proposition. It suffices to show that there are no connected components with exactly two eigenvalues. Suppose to the contrary that there is such a connected component. By Proposition \ref{prop3.7}(2) and the fact that the valency is $k$, it follows that the induced subgraph on this component is either $K_{k+1}$ or $\widetilde{K_k}$. The former one has eigenvalues $k$ and $-1 \notin \{ \sqrt{k-\mu},-\sqrt{k-\mu} \}$ and the latter one has eigenvalue $k$ and  $0 \notin \{ \sqrt{k-\mu},-\sqrt{k-\mu} \}$. The desired contradiction has been reached.
\end{proof}

\begin{remark} \label{Rem3.14}
{\em In an improper LDDG $\Gamma$, every two distinct vertices have a constant number $\lambda$ of neighbours. A similar count as in Proposition \ref{prop3.1} gives $k(k-1)=(v-1)\lambda$, where $k$ is the valency and $v$ is the total number of vertices.

Suppose $\lambda=0$. Then every connected component of $\Gamma$ is either a vertex without loop, a vertex with loop or an edge whose two end vertices do not have loops. The LDDG is then a disjoint union of such graphs and there are at most three eigenvalues, see also Proposition \ref{prop3.7}.

Suppose $\lambda > 0$ and also assume that there are proper edges (otherwise we have the situation again of the previous paragraph). Then $\Gamma$ must be connected with diameter at most 2. If the diameter is 1, then $\Gamma$ is a complete graph $K_l$, $l \geq 3$ or a graph of type $\widetilde{K_l}$, $l \geq 2$. Suppose therefore that the diameter is 2. By Proposition \ref{prop3.6}(2), there are then at least three eigenvalues. If $A$ is an adjacency matrix of $\Gamma$, then $A^2 = k I + \lambda (J-I) = (k-\lambda) I + \lambda J$ by comparing the entries at both sides of the equality. So, $A^2$ has two eigenvalues, namely $k-\lambda$ and $k-\lambda + \lambda v = k^2$, implying that every eigenvalue of $A$ is equal to $k$, $\sqrt{k-\lambda}$, $-\sqrt{k-\lambda}$ or $-k$. If $\Gamma$ has no loops, then it must be a strongly regular graph with equal $\lambda$- and $\mu$-parameters. Suppose therefore that $\Gamma$ has loops. We can then apply the Perron-Frobenius Theorem (\cite[Theorem 2.2.1]{Br-Ha}) to the non-negative matrix $A$. This matrix $A$ is irreducible and primitive, with index equal to $1$. The Perron-Frobenius eigenvalue is equal to $k$ with corresponding Perron eigenvector $[1 \, 1 \, \cdots \, 1]^t$. The Perron-Frobenius Theorem then states that there is only one eigenvalue with absolute value equal to $k$. So, $-k$ cannot be an eigenvalue. Again, there are then exactly three eigenvalues, and Proposition \ref{prop3.13} implies that $\Gamma$ is a modified strongly regular graph for which the $\Gamma$-parameter of the associated strongly regular graph is 2 less that its $\mu$-parameter.}
\end{remark}

\begin{proposition} \label{prop3.15}
Let $\lambda > 0$, $k \geq 2$ and $n \geq 2$ be integers. Let $\mathcal{F}_1$ be the family of all nonisomorphic improper LDDG's with valency $k$ on $n$ vertices for which any two distinct vertices have exactly $\lambda$ common neighbours (see Remark \ref{Rem3.14}). Let $\mathcal{F}_2$ be the family of all nonisomorphic proper LDDG's with valency $k$ and class sizes $n$ for which $\lambda_1=\lambda$ and $\lambda_2=0$. Then any disjoint union of $l_1$ elements of $\mathcal{F}_1$ and $l_2$ elements of $\mathcal{F}_2$ with $l_1+l_2 \geq 2$ and $l_2 \geq 1$ is a disconnected proper LDDG. Conversely, every disconnected proper LDDG can be obtained in this way.
\end{proposition}
\begin{proof}
It is straightforward to verify that any disjoint union like above is a disconnected $k$-regular proper LDDG with class sizes $n$ and parameters $\lambda_1=\lambda$ and $\lambda_2=0$. Conversely, suppose that $\Gamma$ is a disconnected proper LDDG with class sizes $n$, valency $k$ and parameters $\lambda_1$ and $\lambda_2$. As any two vertices belonging to distinct components have no common neighbours, we have $0 \in \{ \lambda_1,\lambda_2 \}$.

We show that $\lambda_2=0$. Suppose to the contrary that $\lambda_1=0$. Then any two vertices belonging to distinct components belong to the same class of the canonical partition. Applying this two times, we see that also vertices belonging to the same component belong to the same class of the canonical partition. This is in contradiction with the fact that there are at least two classes.

So, $\lambda_2=0$ and $\lambda := \lambda_1 > 0$. Let $\mathcal{F}_1$ and $\mathcal{F}_2$ then be the above families for the parameters $\lambda$, $k$ and $n$. Note that a class of the LDDG cannot contain vertices that belong to different components. So, a connected component is a union of classes. If it coincides with a class, then the induced subgraph must belong to $\mathcal{F}_1$. If it is the union of more than one class, then the induced subgraph must belong to $\mathcal{F}_2$. The claim now follows.    
\end{proof}

\subsection{The spectrum of LDDG's} \label{sec3.2}

In this subsection, we assume that $\Gamma$ is a proper LDDG with parameters $(v,k,\lambda_1,\lambda_2,m,n)$, where $m,n \geq 2$ and $\lambda_1 \not= \lambda_2$. We show that $\Gamma$ has at most five eigenvalues, and show how all eigenvalues and their corresponding multiplicities can be computed. The treatment given here is based on Sections 2 and 3 of \cite{HKM}. All matrices considered here are over the reals.

For a positive integer $a$, let $I_a$ and $J_a$ be the $a \times a$ identity matrix and the $a \times a$ all-ones matrix, respectively. We also define 
\[ K_{(m,n)} := I_m \otimes J_n = \mathrm{diag}(J_n,\ldots,J_n). \] 
For every $i \in \{ 2,\ldots,n \}$, let $\eta_i$ be the column matrix
\[   [ -1, \, \cdots \, ,-1, n-1, -1, \, \cdots \, ,-1 ]^t  \]
of size $n$ in which all entries are equal to $-1$, except for the $i$th one which is equal to $n-1$. Let $o_n$ (respectively, $j_n$) be the row matrix of length $n$ with all entries equal to $0$ (respectively $1$). For every $j \in \{ 1,\ldots,m \}$ and every $i \in \{ 2,\ldots,n \}$, let $\omega_{i,j}$ be the column matrix 
\[  [o_n, \, \cdots \, ,o_n, \eta_i, o_n, \, \cdots \, ,o_n]^t  \]
of size $v=mn$ in which all $m$ row blocks of size $n$ are equal to $o_n$, except for the $j$th one which is equal to $\eta_i$. For every $j \in \{ 2,\ldots,m \}$, let $\omega_j$ be the column matrix
\[ [ -j_n, \, \cdots \, ,-j_n,(n-1)j_n,-j_n, \, \cdots \, ,-j_n ]^t   \]
of size $v=mn$ in which all $m$ blocks of size $n$ are equal to $-j_n$ except for the $j$th one which is equal to $(n-1)j_n$. Let $\omega$ be the column matrix of size $v=mn$ in which all entries are equal to 1. The claims in the following lemma are easily verified.

\begin{lemma} \label{lem3.16}
\begin{enumerate}
\item[$(1)$] The column matrices $\omega$, $\omega_j$ with $j \in \{ 2,\ldots,m \}$ and $\omega_{ij}$, with $i \in \{ 2,\ldots,n \}$ and $j \in \{ 1,2,\ldots,m \}$, are mutually orthogonal and generate $\R^v$.
\item[$(2)$] $\omega$ is an eigenvector of $I_v$, $J_v$ and $K_{(m,n)}$ with respective eigenvalues $1$, $v=mn$ and $n$. 
\item[$(3)$] For every $j \in \{ 2,\ldots,m \}$, $\omega_j$ is an eigenvector of $I_v$, $J_v$ and $K_{(m,n)}$ with respective eigenvalues $1$, $0$ and $n$.
\item[$(4)$] For every $i \in \{ 2,\ldots,n \}$ and every $j \in \{ 1,\ldots,m \}$, $\omega_{ij}$ is an eigenvector of $I_v$, $J_v$ and $K_{(m,n)}$ with respective eigenvalues $1$, $0$ and $0$.
\end{enumerate}
\end{lemma}

\begin{proposition} \label{prop3.17}
We have $A^2 = (k-\lambda_1) I_v + \lambda_2 J_v + (\lambda_1-\lambda_2) K_{(m,n)}$.
\end{proposition}
\begin{proof}
This was observed in \cite{HKM} for DDG's, but the argument to show this holds for general LDDG's. If $x$ and $y$ are two vertices, then the number of the walks of length 2 connecting $x$ and $y$ is equal to $k$ if $x=y$, $\lambda_1$ if $x$ and $y$ are two distinct vertices belonging to the same class and $\lambda_2$ if $x$ and $y$ belong to different classes. This implies that $A^2 = k I_v - \lambda_1 (K_{(m,n)} - I_v) + \lambda_2 (J_v - K_{(m,n)}) = (k-\lambda_1) I_v + \lambda_2 J_v + (\lambda_1-\lambda_2) K_{(m,n)}$.
\end{proof}

\bigskip \noindent Using Lemma \ref{lem3.16} and Propositions \ref{prop3.1}, \ref{prop3.17}, we can now compute all eigenvalues of $A^2$ and their corresponding multiplicities. We then see that $\omega$ is an eigenvector of $A^2$ with eigenvalue $(k-\lambda_1) \cdot 1 + \lambda_2 \cdot mn + (\lambda_1 - \lambda_2) \cdot n = k + (n-1) \lambda_1 + \lambda_2 (m-1) n = k^2$, the vectors $\omega_j$, $j \in \{ 2,\ldots,m \}$, are eigenvectors of $A^2$ with eigenvalue $(k-\lambda_1) \cdot 1 + \lambda_2 \cdot 0 + (\lambda_1-\lambda_2) \cdot n = k + (n-1)\lambda - \lambda_2 n = k^2 - \lambda_2 mn = k^2 - \lambda_2 v$, and the vectors $\omega_{ij}$, $i \in \{ 2,\ldots,n \}$ and $j \in \{ 1,\ldots,m \}$, are eigenvectors of $A^2$ with eigenvalue $(k - \lambda_1) \cdot 1 + \lambda_2 \cdot 0 + (\lambda_1- \lambda_2) \cdot 0 = k- \lambda_1$. So, we have found a complete set of eigenvectors for the matrix $A^2$ and their corresponding eigenvalues. As all eigenvalues of $A^2$ are nonnegative integers, we have that $k^2 - \lambda_2 v \geq 0$ and $k - \lambda_1 \geq 0$ (but we already knew the latter, see Proposition \ref{prop3.4}). We thus have the following improvement of the upper bound in Proposition \ref{prop3.5} (also note that $k \leq v-2$ by Proposition \ref{prop3.3}).

\begin{proposition} \label{prop3.18}
We have $\lambda_2 \leq \lfloor \frac{k^2}{v} \rfloor$.
\end{proposition}

\medskip \noindent Now, put $\mu_0 :=k$, $\mu_1 := \sqrt{k-\lambda_1}$, $\mu_2 := -\mu_1$, $\mu_3 := \sqrt{k^2 - \lambda_2 v}$ and $\mu_4 := - \mu_3$. Let $E_0$ be the subspace generated by the eigenvector $\omega$. The space generated by the vectors $\omega_{ij}$, $i \in \{ 2,\ldots,n \}$ and $j \in \{ 1,\ldots,m \}$, can be written as an orthogonal direct sum $E_1 \oplus E_2$ such that every nonzero vector $E_i$, $i \in \{ 1,2 \}$, is an eigenvector of $A$ for the eigenvalue $\mu_i$. Note that if $\mu_1=-\mu_2 \not= 0$, then this decomposition is unique. Similarly, the space generated by the vectors $\omega_j$, $j \in \{ 2,\ldots,m \}$, can be written as an orthogonal direct sum $E_3 \oplus E_4$ such that every nonzero vector of $E_i$, $i \in \{ 3,4 \}$, is an eigenvector of $A$ for the eigenvalue $\mu_i$. Again, if $\mu_3=-\mu_4 \not= 0$, then this decomposition is unique. We thus have the orthogonal direct sum
\[ \R^v = E_0 \oplus E_1 \oplus E_2 \oplus E_3 \oplus E_4,  \]
where 
\begin{itemize}
\item $\dim(E_0)=1$, $\dim(E_1 \oplus E_2)=mn-m$, $\dim(E_3 \oplus E_4)=m-1$,
\item every nonzero vector of $E_i$, $i \in \{ 1,2,3,4 \}$, is an eigenvector of $A$ with eigenvalue $\mu_i$.
\end{itemize}
If we put $f_i:=\dim(E_i)$, $i \in \{ 0,1,2,3,4 \}$, then we have $f_0=1$, $f_1+f_2=mn-m$ and $f_3+f_4=m-1$. As the trace of a matrix equals the sum of its eigenvalues, taking into account their multiplicities, and the trace of $A$ equals the number $L$ of loops in $\Gamma$, the following holds.

\begin{lemma} \label{lem3.19}
We have $L=k + (f_1-f_2) \cdot \sqrt{k-\lambda_1} + (f_3 - f_4) \cdot \sqrt{k^2-\lambda_2 v}$.
\end{lemma}

\medskip \noindent Let $X_1,X_2,\ldots,X_m$ denote the classes of the canonical partition. For all $i,j \in \{ 1,2,\ldots,m \}$, let $N_{ij}$ denote the number of ordered pairs $(u,v) \in X_i \times X_j$ for which $u \sim v$. Let $N^\ast$ denote the number of ordered pairs $(u,v)$ of vertices of $\Gamma$ for which $u \sim v$ and $u$, $v$ belong to the same class. Note that $N^\ast = N_{11} + N_{22} + \cdots + N_{mm}$. Let $R$ denote the $m \times m$ real symmetric matrix whose $(i,j)$th entry is equal to $\frac{1}{n} N_{ij}$. The following was proved in \cite[Theorem 3.1]{HKM}.

\begin{lemma}[\cite{HKM}] \label{lem3.20}
We have $R^2 = (k^2 - \lambda_2 v) I_m + \lambda_2 n J_m$. The possible eigenvalues of $R$ are $\mu_0=k$, $\mu_3 = \sqrt{k^2-\lambda_2 v}$ and $\mu_4 = - \sqrt{k^2-\lambda_2v}$. The vector $j_m^t$ is an eigenvector of $R$ for the eigenvalue $k$. If $k^2-\lambda_2 v \not= 0$, then the multiplicities of $\mu_3$ and $\mu_4$ as eigenvalues of $R$ inside $j_m^\perp$ coincide with the multiplicities $f_3$ and $f_4$ of $\mu_3$ and $\mu_4$ as eigenvalues of $A$ inside the space $E_3 \oplus E_4$.  
\end{lemma}

\medskip \noindent In fact, in \cite{HKM}, it was observed that if $u$ is an eigenvector of $R$, then $S u$ with $S = I_m \otimes j_n^t$ is an eigenvector of $A$ for the same eigenvalue. Note that such an eigenvector $S u$ necessarily belongs to the subspace $E_0 \oplus E_3 \oplus E_4$, explaining the latter claim in Lemma \ref{lem3.20}. The trace of $R$ equals $\frac{N^\ast}{n}$ and is also equal to the sum of its eigenvalues, taking into account their multiplicities. We thus have

\begin{corollary} \label{co3.21}
We have $(f_3 - f_4 )\cdot \sqrt{k^2-\lambda_2 v} = -k + \frac{N^\ast}{n}$.
\end{corollary}

\medskip \noindent If $k^2 \not= \lambda_2 v$, then Corollary \ref{co3.21} in combination with $f_3 + f_4 = m-1$ allows to determine $f_3$ and $f_4$. If $k^2=\lambda_2 v$, then the multiplicity of $\mu_3=\mu_4=0$ as eigenvalue of $A$ inside $E_3 \oplus E_4$ is equal to $m-1$ (in this case, only the value $f_3+f_4$ is relevant, and not those of the individual $f_3$ and $f_4$). Corollary \ref{co3.21} implies the following:

\begin{corollary} \label{co3.22}
If $k^2 - \lambda_2 v$ is not a square, then $m$ is odd, $f_3=f_4=\frac{m-1}{2}$ and $N^\ast = kn$.
\end{corollary}

\medskip \noindent Lemma \ref{lem3.19}, Corollaries \ref{co3.21}, \ref{co3.22} and the fact that $f_1+f_2=(m-1)n$ also imply the following.

\begin{corollary} \label{co3.23}
\begin{enumerate}
\item[$(a)$] We have $L=\frac{N^\ast}{n} + (f_1-f_2) \cdot \sqrt{k-\lambda_1}$.
\item[$(b)$] If $k-\lambda_1$ is not a square, then $(m-1)n$ is even, $f_1=f_2=\frac{(m-1)n}{2}$ and $N^\ast = nL$.
\item[$(c)$] If neither of $k-\lambda_1$ and $k^2 - \lambda_2 v$ is a square, then the number $L$ of loops inside $\Gamma$ equals its valency $k$.  
\end{enumerate}
\end{corollary}

\medskip \noindent If $k \not= \lambda_1$, then Corollary \ref{co3.23}(a) in combination $f_1+f_2=(m-1)n$ allows to determine $f_1$ and $f_2$. If $k=\lambda_1$, then the multiplicity of $\mu_1=\mu_2=0$ as eigenvalue of $A$ inside $E_1 \oplus E_2$ is equal to $f_1+f_2=(m-1)n$. Again, in this case, only the sum $f_1+f_2$ is then relevant and not the individual values of $f_1$ and $f_2$. Note that for a proper DDG (without loops), Corollary \ref{co3.23}(c) implies that at least one of $k-\lambda_1$, $k^2 - \lambda_2 v$ is a square. This was already shown in Theorem 2.2(a) of \cite{HKM}.

\subsection{LDDG's and automorphisms of order 2} \label{sec3.3}

In this subsection, we show how an automorphism of order 2 in an LDDG results in another (potentially nonisomorphic) LDDG. Let us start from a general undirected graph $\Gamma$ with vertex set $\{ 1,2,\ldots,v \}$, in which loops are allowed, and let $A$ be the adjacency matrix of $\Gamma$ with respect to the natural ordering of the vertices. The total number of loops in $\Gamma$ is then equal to the trace of $A$. The following is a straightforward translation of the conditions for $\Gamma$ to be an LDDG.

\begin{lemma} \label{lem3.24}
The graph $\Gamma$ is an LDDG with parameters $(v,k,\lambda_1,\lambda_2,m,n)$ if the following conditions hold:
\begin{enumerate}
\item[$(1)$] $A$ is a (symmetric) $(v \times v)$-matrix;
\item[$(2)$] in each row (and column) of $A$, there are exactly $k$ entries that are equal to $1$;
\item[$(3)$] there exists a partition $\{ X_1,X_2,\ldots,X_m \}$ of $\{ 1,2,\ldots,v \}$ in $m$ subsets of size $n$ for which the following hold:
\begin{itemize}
\item if $i$ and $j$ are two distinct elements of $\{ 1,2,\ldots,v \}$ belonging to the same $X_l$, then there are exactly $\lambda_1$ positions in which the entries in both rows are equal to $1$;  
\item if $i$ and $j$ are two elements of $\{ 1,2,\ldots,v \}$ belonging to distinct $X_l$'s, then there are exactly $\lambda_2$ positions in which the entries in both rows are equal to $1$.
\end{itemize}
\end{enumerate}
\end{lemma}

\medskip \noindent The properties (1), (2) and (3) remain invariant under row permutations, except for the condition that $A$ is symmetric. This implies the following.

\begin{lemma} \label{lem3.25}
Suppose $P$ is a permutation matrix and $B$ is a $v \times v$-matrix with all entries equal to $1$ such that both $B$ and $PB$ are symmetric matrices. Let $\Omega$ and $\Omega'$ be the undirected graphs for which $B$ and $PB$ are the respective adjacency matrices. Then $\Omega$ is an LDDG if and only if $\Omega'$ is an LDDG, in which case the parameters of both LDDG's are the same. 
\end{lemma}
 
\medskip \noindent We show that suitable permutation matrices $P$ arise from certain automorphisms of order $2$ of $\Gamma$. Suppose $\theta$ is a permutation of $\{ 1,2,\ldots,v \}$. Then associated with $\theta$, there is a permutation matrix $P$ defined by
\begin{displaymath}
P_{ij} := \left\{
\begin{array}{ll}
1 & \mbox{if } j=i^\theta, \\
0 & \mbox{otherwise}, 
\end{array}
\right.
\end{displaymath}
for elements $i,j \in \{ 1,2,\ldots,v \}$. The following gives a sufficient conditions for $P_\theta A$ to be a symmetric matrix.

\begin{lemma} \label{lem3.26}
If $\theta$ is an automorphism of $\Gamma$ such that $\theta^2$ is the identity, then $P_\theta A$ is a symmetric matrix and hence the adjacency matrix of a graph $\Gamma_\theta$.
\end{lemma}
\begin{proof}
We make use of the following known fact:
\begin{quote}
$\theta$ is an automorphism of $\Gamma$ if and only if $P_\theta A P_\theta^{-1}=A$.
\end{quote}
As $\theta^2=1$, the condition $j=i^\theta$ implies that $i=j^\theta$ for $i,j \in \{ 1,2,\ldots,v \}$, implying that $P_\theta = P_\theta^t$. We then have $(P_\theta A)^t = A^t P_\theta^t = A P_\theta$. By the above $AP_\theta = P_\theta A$ as $\theta$ is an automorphism. So, $P_\theta A$ is indeed a symmetric matrix. 
\end{proof}

\begin{lemma} \label{lem3.27}
Suppose $\theta$ is an automorphism of $\Gamma$ for which $\theta^2$ is the identity. Let $N$ denote the number of loops of $\Gamma$. Let $A_1$ denote the number of vertices $v$ without loops for which $v^\theta \sim_\Gamma v$, and let $A_2$ denote the number of vertices $w$ with loops for which $w^\theta \not\sim_\Gamma w$. Then the number of loops in $\Gamma_\theta$ is equal to $N+A_1-A_2$.
\end{lemma}
\begin{proof}
The number of loops of $\Gamma_\theta$ is equal to
\[ \sum_{i=1}^n (P_\theta A)_{ii} = \sum_{i=1}^n \sum_{l=1}^n (P_\theta)_{il} A_{li} = \sum_{i=1}^n A_{i^\theta i} = {\sum}' A_{i^\theta i} + {\sum}'' A_{i^\theta i},  \]
where $\Sigma'$ is the summation over all vertices $i$ with a loop and $\Sigma''$ is the summation over all vertices $i$ without a loop. The first sum is equal to $N-A_2$, while the second sum is equal to $A_1$.  
\end{proof} 
 
\medskip \noindent Summarizing, we have proved the following.
 
\begin{theorem} \label{theo3.28}
Let $\Gamma$ be an LDDG with parameters $(v,k,\lambda_1,\lambda_2,m,n)$ and suppose $\theta$ is an automorphism of order $2$ of $\Gamma$. Then define a graph $\Gamma_\theta$ on the same vertex set $V$ as $\Gamma$, by calling two (possibly equal) vertices $x$ and $y$ adjacent in $\Gamma_\theta$ if and only if $x^\theta$ and $y$ are adjacent in $\Gamma$. Then the following hold.
\begin{enumerate}
\item[$(1)$] $\Gamma_\theta$ is an LDDG with the same parameters as $\Gamma$.
\item[$(2)$] Let $N$ denote the number of loops of $\Gamma$, $A_1$ the number of vertices $x$ without loops in $\Gamma$ for which $x^\theta \sim_\Gamma x$ and $A_2$ the number of vertices $y$ with loops in $\Gamma$ for which $y^\theta \not\sim_\Gamma y$. Then the number of loops in $\Gamma_\theta$ is equal to $N+A_1-A_2$. 
\end{enumerate}
\end{theorem}
\begin{proof}
Note that $(P_\theta A)_{ij} = A_{i^\theta j}$. So, in the LDDG defined by the symmetric matrix $P_\theta A$, the vertices $i$ and $j$ are adjacent if and only if $i^\theta$ and $j$ are adjacent in the original LDDG $\Gamma$. The claim now follows from Lemmas \ref{lem3.25}, \ref{lem3.26} and \ref{lem3.27}.
\end{proof}

\begin{remark} \label{Rem3.29}
{\em The operation of permuting rows, but not columns, of an adjacency matrix is called {\em dual Seidel switching}. This operation was introduced by Willem Haemers in 1984. In Lemma \ref{lem3.26}, we apply this operation. Note that dual Seidel switching applied to a Deza graph (resp. an LDDG) results in a Deza graph (resp. an LDDG) with the same parameters; see e.g. Section 1.4.1 of \cite{Go-Sh}.}
\end{remark}

\begin{remark} \label{Rem3.30}
{\em The graph operations $\Gamma \mapsto \overline{\Gamma}$ and $\Gamma \mapsto \Gamma_\theta$ for order $2$ automorphisms $\theta$ produce LDDG's from other LDDG's (Lemma \ref{lem1.1}, Theorem \ref{theo3.28}), and both operations commute. If $\overline{A_1}$ and $\overline{A_2}$  are similarly defined as $A_1$ and $A_2$ (cfr. Theorem \ref{theo3.28}), but for the graph $\overline{\Gamma}$ instead of $\Gamma$, then $\overline{A_1} = A_2$ and $\overline{A_2} = A_1$.}
\end{remark}

\begin{remark} \label{Rem3.31}
{\em In Sections \ref{sec4} and \ref{sec5}, we will construct several families of LDDG's. Certain of these LDDG's admit automorphisms $\theta$ of order 2 induced by symplectic and unitary transvections of vector spaces and/or elations and homologies of projective spaces. We have verified that in certain of these cases the resulting LDDG $\Gamma_\theta$ is nonisomorphic to the original LDDG $\Gamma$, but that it still belonged to one of the families defined in Sections \ref{sec4} or \ref{sec5}. It is possible that Theorem \ref{theo3.28} can ever be used to construct entirely new LDDG's.}
\end{remark}

\section{A first family of LDDG's} \label{sec4}

Consider in the projective space $\PG(m-1,q)$, $m \geq 3$, a polarity $\zeta$. Let $\pi$ be a subspace of dimension $k-1 \geq 0$ which is totally isotropic with respect to $\zeta$. Then $\pi^\zeta$ is an $(m-k-1)$-dimensional subspace containing $\pi$. Let $V$ denote the set of points of $\PG(m-1,q)$ not contained in $\pi^\zeta$. Then
\[ |V| = \frac{q^m-1}{q-1} - \frac{q^{m-k}-1}{q-1} = \frac{q^{m-k}(q^k-1)}{q-1}. \]
The graph $\Gamma$ defined on the vertex set $V$, by calling two (not necessarily distinct) elements $x,y \in V$ adjacent whenever $y \in x^\zeta$ is an undirected graph with possible loops.

\begin{lemma} \label{lem4.1}
For every $x \in V$, $x^\zeta$ does not contain $\pi^\zeta$, and $|x^\zeta \cap \pi^\zeta| = \frac{q^{m-k-1}-1}{q-1}$.
\end{lemma}
\begin{proof}
As $x \not\in \pi$, the subspace $\pi^\zeta$ is not contained in $x^\zeta$, implying that $x^\zeta \cap \pi^\zeta$ is a hyperplane of $\pi^\zeta$, having dimension $m-k-2$ and containing $\frac{q^{m-k-1}-1}{q-1}$ points.
\end{proof}

\begin{lemma} \label{lem4.2}
The graph $\Gamma$ is regular with valency $\frac{q^{m-k-1}(q^k-1)}{q-1}$.
\end{lemma}
\begin{proof}
The number of vertices adjacent to a given vertex $x$ is equal to
\[ |x^\zeta \setminus \pi^\zeta| = |x^\zeta| - |x^\zeta \cap \pi^\zeta| = \frac{q^{m-1}-1}{q-1}-\frac{q^{m-k-1}-1}{q-1} = \frac{q^{m-k-1}(q^k-1)}{q-1}. \]
\end{proof}

\begin{lemma} \label{lem4.3}
If $x$ and $y$ are two vertices of $V$, then $x^\zeta \cap \pi^\zeta = y^\zeta \cap \pi^\zeta$ if and only if $\langle x,\pi \rangle = \langle y,\pi \rangle$.
\end{lemma}
\begin{proof}
Indeed, $x^\zeta \cap \pi^\zeta = y^\zeta \cap \pi^\zeta$ if and only if $(x^\zeta \cap \pi^\zeta)^\zeta = \langle x,\pi \rangle$ equals $(y^\zeta \cap \pi^\zeta)^\zeta = \langle y,\pi \rangle$.
\end{proof}

\medskip \noindent We can now partition the vertex set $V$ into $\frac{q^{m-k}-q^{m-2k}}{q-1}$ classes of size $q^k$ such that two vertices $x$ and $y$ belong to the same class if and only if $\langle x,\pi \rangle = \langle y,\pi \rangle$.

\begin{lemma} \label{lem4.4}
\begin{enumerate}
\item[$(a)$] If $x$ and $y$ are two distinct vertices of the same class, then $x$ and $y$ have exactly $\frac{q^{m-k-1}(q^{k-1}-1)}{q-1}$ common neighbours.
\item[$(b)$] If $x$ and $y$ are two vertices belonging to different classes, then $x$ and $y$ have exactly $\frac{q^{m-k-2}(q^k-1)}{q-1}$ common neighbours.
\end{enumerate}
\end{lemma}
\begin{proof}
The total number of common neighbours is equal to
\[ |(x^\zeta \cap y^\zeta) \setminus \pi^\zeta| = |x^\zeta \cap y^\zeta| - |x^\zeta \cap y^\zeta \cap \pi^\zeta| = \frac{q^{m-2}-1}{q-1} - |x^\zeta \cap y^\zeta \cap \pi^\zeta|. \]
If $x$ and $y$ belong to the same class, then $x^\zeta \cap \pi^\zeta = y^\zeta \cap \pi^\zeta = x^\zeta \cap y^\zeta \cap \pi^\zeta$ contains $\frac{q^{m-k-1}-1}{q-1}$ points by Lemmas \ref{lem4.1} and \ref{lem4.3}. If $x$ and $y$ belong to different classes, then $x^\zeta \cap \pi^\zeta \not= y^\zeta \cap \pi^\zeta$ and so $x^\zeta \cap y^\zeta \cap \pi^\zeta$ has dimension $m-k-3$ and contains $\frac{q^{m-k-2}-1}{q-1}$ points. In the former case, the total number $|\Gamma(x) \cap \Gamma(y)|$ of common neighbours of $x$ and $y$ is thus equal to $\frac{q^{m-2}-1}{q-1} - \frac{q^{m-k-1}-1}{q-1} = \frac{q^{m-k-1}(q^{k-1}-1)}{q-1}$, while in the latter case, this number is equal to $\frac{q^{m-2}-1}{q-1} - \frac{q^{m-k-2}-1}{q-1} = \frac{q^{m-k-2}(q^k-1)}{q-1}$.
\end{proof}

\medskip \noindent We conclude the following.

\begin{corollary} \label{co4.5}
The graph $\Gamma$ is an LDDG with parameters equal to
\[ (\frac{q^{m-k}(q^k-1)}{q-1},\frac{q^{m-k-1}(q^k-1)}{q-1},\frac{q^{m-k-1}(q^{k-1}-1)}{q-1},\frac{q^{m-k-2}(q^k-1)}{q-1},\frac{q^{m-k}-q^{m-2k}}{q-1},q^k). \]
\end{corollary}

\begin{corollary} \label{co4.6}
The complementary graph $\widetilde{\Gamma}$ of $\Gamma$ is an LDDG with parameters equal to 
\[ (\frac{q^{m-k}(q^k-1)}{q-1},q^{m-k-1}(q^k-1),q^{m-1}-q^{m-2}-q^{m-k-1},q^{m-k-2}(q-1)(q^k-1),\frac{q^{m-k}-q^{m-2k}}{q-1},q^k). \]
\end{corollary}

\medskip \noindent For each of the types of polarities, we will now specify what these parameters are, and indicate how many loops the LDDG's $\Gamma$ and $\widetilde{\Gamma}$ have. There will be two instances where no loops are present.  

\begin{example} \label{Ex4.7}
{\em Suppose $\zeta$ is a symplectic polarity of $\PG(2n-1,q)$, $n \geq 2$, and $\pi$ is a subspace of dimension $n-s-1$, $s \in \{ 0,1,\ldots,n-1 \}$, which is totally isotropic with respect to $\zeta$. Then $\Gamma$ is an LDDG with parameters equal to
\[ (\frac{q^{n+s}(q^{n-s}-1)}{q-1},\frac{q^{n+s-1}(q^{n-s}-1)}{q-1},\frac{q^{n+s-1}(q^{n-s-1}-1)}{q-1},\frac{q^{n+s-2}(q^{n-s}-1)}{q-1},\frac{q^{n+s}-q^{2s}}{q-1},q^{n-s}).  \]
As $x \in x^\zeta$ for every vertex $x \in V$, all vertices of $\Gamma$ have loops. The complementary graph $\widetilde{\Gamma}$ of $\Gamma$ is an LDDG with parameters
\[ (\frac{q^{n+s}(q^{n-s}-1)}{q-1},q^{n+s-1}(q^{n-s}-1),q^{2n-1}-q^{2n-2}-q^{n+s-1},q^{n+s-2}(q-1)(q^{n-s}-1),\frac{q^{n+s}-q^{2s}}{q-1},q^{n-s}),  \]
and has no loops.}
\end{example}

\begin{example} \label{Ex4.8}
{\em Suppose $\zeta$ is a pseudo-polarity in $\PG(2n-1,q)$, $n \geq 2$. Then the set of all points $x$ for which $x \in x^\zeta$ is a hyperplane $H^\ast$ of $\PG(2n-1,q)$ and $(H^\ast)^\zeta$ is a point $x^\ast$ contained in $H^\ast$. Let $\pi$ be a subspace of dimension $n-s-1$, $s \in \{ 0,1,\ldots,n-1 \}$, which is totally isotropic with respect to $\zeta$. Then $\pi \subseteq H^\ast$ and $\Gamma$, $\widetilde{\Gamma}$ are LDDG's with the same parameters as in the symplectic case (Example \ref{Ex4.7}). 

Suppose $\pi$ does not contain $x^\ast$. Then $s \geq 1$ and $\pi^\zeta$ intersects $H^\ast$ in a subspace of dimension $n+s-2$. In this case, the graph $\Gamma$ has $\frac{q^{2n-1}-q^{n+s-1}}{q-1} \geq \frac{q^{2n-1}-q^{2n-2}}{q-1} > 0$ loops and $\widetilde{\Gamma}$ has $q^{2n-1}-q^{n+s-1} > q^{2n-1}-q^{2n-2} > 0$ loops.

Suppose $\pi$ contains $x^\ast$. Then $\pi^\zeta \cap H^\ast=\pi^\zeta$ has dimension $n+s-1$. The graph $\widetilde{\Gamma}$ has $q^{2n-1} > 0$ loops and $\Gamma$ has $\frac{q^{2n-1}-q^{n+s}}{q-1}$ loops. The latter number is 0 if and only if $s=n-1$, i.e. if and only if $\pi$ coincides with the point $x^\ast$. In this case, $\Gamma$ is a DDG having parameters 
\[  (q^{2n-1},q^{2n-2},0,q^{2n-3},q^{2n-2},q). \]
}
\end{example}

\begin{example} \label{Ex4.9}
{\em Suppose $\zeta$ is a pseudo-polarity in $\PG(2n,q)$, $n \geq 1$. The set of all points $x$ for which $x \in x^\zeta$ is a hyperplane $H^\ast$ of $\PG(2n,q)$ and $(H^\ast)^\zeta$ is a point $x^\ast$ not contained in $H^\ast$. Let $\pi$ be a subspace of dimension $n-s-1$, $s \in \{ 0,1,\ldots,n-1 \}$, which is totally isotropic with respect to $\zeta$. Then $\pi \subseteq H^\ast$ and $\Gamma$ is an LDDG with parameters equal to
\[ (\frac{q^{n+s+1}(q^{n-s}-1)}{q-1},\frac{q^{n+s}(q^{n-s}-1)}{q-1},\frac{q^{n+s}(q^{n-s-1}-1)}{q-1},\frac{q^{n+s-1}(q^{n-s}-1)}{q-1},\frac{q^{n+s+1}-q^{2s+1}}{q-1},q^{n-s}). \]
On the other hand, $\widetilde{\Gamma}$ is an LDDG with parameters
\[  (\frac{q^{n+s+1}(q^{n-s}-1)}{q-1},q^{n+s}(q^{n-s}-1),q^{2n}-q^{2n-1}-q^{n+s},q^{n+s-1}(q-1)(q^{n-s}-1),\frac{q^{n+s+1}-q^{2s+1}}{q-1},q^{n-s}). \]
Note that $\pi^\zeta$ intersects $H^\ast$ in a subspace of dimension $n+s-1$. The graph $\Gamma$ therefore has $\frac{q^{2n}-q^{n+s}}{q-1} > 0$ loops and the graph $\widetilde{\Gamma}$ has $q^{2n}-q^{n+s} > 0$ loops.}
\end{example}

\begin{example} \label{Ex4.10}
{\em Suppose $\zeta$ is an orthogonal polarity in $\PG(m-1,q)$ with $m \geq 3$ and $q$ odd. Let $\pi$ be a subspace of dimension $k-1 \geq 0$ which is totally isotropic with respect to $\zeta$. Let $Q$ be the set of all points $x$ for which $x \in x^\zeta$. We distinguish three cases: 
\begin{enumerate}
\item[(a)] $m$ is odd and $Q$ is a parabolic quadric. Then $k \in \{ 1,\ldots,\frac{m-1}{2} \}$. 
\item[(b)] $m$ is even and $Q$ is a hyperbolic quadric. Then $k \in \{ 1,\ldots,\frac{m}{2} \}$.
\item[(c)] $m$ is even and $Q$ is an elliptic quadric. Then $k \in \{ 1,\ldots,\frac{m-2}{2} \}$.
\end{enumerate}
The parameters of the LDDG's $\Gamma$ and $\widetilde{\Gamma}$ are as in Corollaries \ref{co4.5} and \ref{co4.6}. We now determine the number of loops in these graphs. The number of loops in $\Gamma$ is equal to $|Q|-|Q \cap \pi^\zeta|$ and the number of loops in $\widetilde{\Gamma}$ is $\frac{q^{m-k}(q^k-1)}{q-1}$ minus that number. So, it suffices to determine the former number.  Note that since the nonsingular quadric $Q$ cannot be contained in a proper subspace, $\Gamma$ must have loops. As $|Q|$ is strictly smaller than the size of the complement of any proper subspace of $\PG(m-1,q)$ also $\widetilde{\Gamma}$ must have loops. We distinguish between the various cases.
\begin{itemize}
\item Suppose $m=2n+1$ is odd and $\pi$ be a subspace of dimension $n-s-1$ with $n \geq 1$ and $s \in \{ 0,1,\ldots,n-1 \}$. Then $\pi^\zeta$ has dimension $n+s$ and intersects $Q$ in a cone of type $\pi Q(2s,q)$. This cone contains $\frac{q^{n+s}-1}{q-1}$ points, while $Q$ itself contains $\frac{q^{2n}-1}{q-1}$ points. So, $\Gamma$ has $q^{n+s} \frac{q^{n-s}-1}{q-1}$ loops.
\item Suppose $m=2n \geq 2$ is even. Then $|Q|= \frac{q^{2n-1}-1}{q-1} + \epsilon q^{n-1}$, where $\epsilon = +1$ if $Q$ is a hyperbolic quadric and $\epsilon=-1$ if $Q$ is an elliptic quadric. Suppose $\pi$ is a totally isotropic subspace of dimension $n-s-1$, where $s \in \{ 0,1,\ldots,n-1 \}$ if $\epsilon = +1$ and $s \in \{ 1,2,\ldots,n-1 \}$ if $\epsilon = -1$. Then $\pi^\zeta$ has dimension $n+s-1$ and intersects $Q$ in a cone of type $\pi Q^\epsilon(2s-1,q)$. This cone contains $\frac{q^{n+s-1}-1}{q-1} + \epsilon q^{n-1}$ points. So, $\Gamma$ has $\frac{q^{n+s-1}(q^{n-s}-1)}{q-1}$ loops.
\end{itemize}
}
\end{example}

\begin{example} \label{Ex4.11}
{\em Suppose $\zeta$ is a unitary polarity in $\PG(m-1,q^2)$ with $m \geq 3$ (for convenience, we replace here $q$ by $q^2$). Let $\pi$ be a subspace of dimension $k-1 \geq 0$ which is totally isotropic with respect to $\zeta$. Let $\mathcal{H}$ denote the nonsingular Hermitean variety consisting of all points $x$ for which $x \in x^\zeta$. We now determine the number of loops in $\Gamma$ and $\widetilde{\Gamma}$. The number of loops in $\Gamma$ is equal to $|\mathcal{H}| - |\mathcal{H} \cap \pi^\zeta|$ and the number of loops in $\widetilde{\Gamma}$ is equal to $\frac{q^{2(m-k)}(q^{2k}-1)}{q^2-1}$ minus that number. So, it suffices to determine the former number. Note that since $\mathcal{H}$ cannot be contained in a proper subspace, $\Gamma$ must have loops. As $|\mathcal{H}|$ is strictly smaller than the size of the complement of any proper subspace of $\PG(m-1,q^2)$, also $\widetilde{\Gamma}$ must have loops. We distinguish between two cases.    
\begin{enumerate}
\item[(a)] Suppose $m=2n$ is even and $k=n-s$, with $s \in \{ 0,1,\ldots,n-1 \}$. Then $|\mathcal{H}| = \frac{(q^{2n}-1)(q^{2n-1}+1)}{q^2-1}$. The subspace $\pi^\zeta$ has dimension $n+s-1$ and intersects $\mathcal{H}$ in a cone of type $\pi H(2s-1,q^2)$. This cone contains $\frac{q^{2n+2s-1}+q^{2n}-q^{2n-1}-1}{q^2-1}$ points. So, $\Gamma$ has $\frac{q^{2n+2s-1}(q^{2(n-s)}-1)}{q^2-1}$ loops.
\item[(b)] Suppose $m=2n+1$ is odd and $k=n-s$, with $s \in \{ 0,1,\ldots,n-1 \}$. Then $|\mathcal{H}|=\frac{(q^{2n+1}+1)(q^{2n}-1)}{q^2-1}$. The subspace $\pi^\zeta$ has dimension $n+s$ and intersects $\mathcal{H}$ in a cone of type $\pi H(2s,q^2)$. This cone contains $\frac{q^{2n+2s+1}-q^{2n+1}+q^{2n}-1}{q^2-1}$ points. So, $\Gamma$ has $\frac{q^{2n+2s+1}(q^{2(n-s)}-1)}{q^2-1}$ loops.
\end{enumerate}
}
\end{example}

\section{A second family of LDDG's} \label{sec5}

Consider in the projective space $\PG(m-1,q)$, $m,q \geq 3$, a polarity $\zeta$ and suppose $x^\ast$ is a point of $\PG(m-1,q)$ not contained in the hyperplane $H^\ast := (x^\ast)^\zeta$.

Let $V$ denote the set of points of $\PG(m-1,q) \setminus (H^\ast \cup \{ x^\ast \})$. Then
\[ |V| = \frac{q^m-1}{q-1} - \frac{q^{m-1}-1}{q-1}-1=q^{m-1}-1.  \]

Let $\Gamma$ be the graph defined on the vertex set $V$, by calling two possibly equal vertices $x,y$ adjacent whenever $y \in x^\zeta$.

\begin{lemma} \label{lem5.1}
For every $x \in V$, $x^\zeta$ does not contain $x^\ast$.
\end{lemma}
\begin{proof}
This follows from the fact that $x$ is not contained in $(x^\ast)^\zeta=H^\ast$.
\end{proof}

\begin{lemma} \label{lem5.2}
For every $x \in V$, $x^\zeta$ and $H^\ast$ are distinct hyperplanes and $|x^\zeta \cap H^\ast| = \frac{q^{m-2}-1}{q-1}$.
\end{lemma}
\begin{proof}
As $x \not= x^\ast$, the hyperplanes $x^\zeta$ and $(x^\ast)^\zeta=H^\ast$ are distinct and intersect in a subspace of dimension $m-3$.
\end{proof}
 
\begin{lemma} \label{lem5.3}
The graph $\Gamma$ is regular with valency $q^{m-2}$.
\end{lemma}
\begin{proof}
Using Lemmas \ref{lem5.1} and \ref{lem5.2}, we know that the number of vertices adjacent to a given vertex $x \in V$ is equal to
\[ |x^\zeta \setminus H^\ast| = |x^\zeta \setminus (x^\zeta \cap H^\ast)| = |x^\zeta|-|x^\zeta \cap H^\ast| = \frac{q^{m-1}-1}{q-1} - \frac{q^{m-2}-1}{q-1} = q^{m-2}.  \]
\end{proof}

\begin{lemma} \label{lem5.4}
If $x$ and $y$ are two distinct vertices of $V$, then $x^\zeta \cap H^\ast = y^\zeta \cap H^\ast$ if and only if $x^\ast$ is contained on the line through $x$ and $y$.
\end{lemma}
\begin{proof}
Indeed, $x^\zeta \cap H^\ast = y^\zeta \cap H^\ast$ if and only if $(x^\zeta \cap H^\ast)^\zeta = \langle x,x^\ast \rangle$ coincides with $y^\zeta \cap H^\ast = \langle y,x^\ast \rangle$.
\end{proof}

\medskip \noindent We can now partition the vertex set $V$ into $\frac{q^{m-1}-1}{q-1}$ classes of size $q-1$, where each such class is obtained from one of the  $\frac{q^{m-1}-1}{q-1}$ lines through $x^\ast$ by removing the points that are contained in $\{ x^\ast \} \cup H^\ast$.

\begin{lemma} \label{lem5.5}
\begin{enumerate}
\item[$(a)$] Two distinct vertices of the same class have no common neighbours.
\item[$(b)$] Two vertices belonging to distinct classes have $q^{m-3}$ common neighbours.
\end{enumerate}
\end{lemma}
\begin{proof}
By Lemma \ref{lem5.1}, the number of common neighbours of two distinct vertices $x$ and $y$ is equal to
\[ |(x^\zeta \cap y^\zeta) \setminus H^\ast| = |x^\zeta \cap y^\zeta| - |x^\zeta \cap y^\zeta \cap H^\ast|. \]
If $x$ and $y$ belong to the same class, then $x^\zeta \cap H^\ast = y^\zeta \cap H^\ast$ by Lemma \ref{lem5.4}. We then have $x^\zeta \cap y^\zeta = x^\zeta \cap y^\zeta \cap H^\ast$ and $x$, $y$ have no common neighbours.

If $x$ and $y$ belong to distinct classes, then $x^\zeta \cap H^\ast \not= y^\zeta \cap H^\ast$ and $x^\zeta \cap y^\zeta \cap H^\ast$ has dimension $m-4$ containing $\frac{q^{m-3}-1}{q-1}$ points. In this case, the number $|\Gamma(x) \cap \Gamma(y)|$ of common neighbours of $x$ and $y$ is equal to $|x^\zeta \cap y^\zeta| - |x^\zeta \cap y^\zeta \cap H^\ast| = \frac{q^{m-2}-1}{q-1}-\frac{q^{m-3}-1}{q-1}=q^{m-3}$.
\end{proof}

\begin{corollary} \label{co5.6}
The graph $\Gamma$ is an LDDG with parameters equal to
\[ (q^{m-1}-1,q^{m-2},0,q^{m-3},\frac{q^{m-1}-1}{q-1},q-1). \] 
\end{corollary}

\begin{corollary} \label{co5.7}
The graph $\widetilde{\Gamma}$ is an LDDG with parameters equal to
\[ (q^{m-1}-1,q^{m-1}-q^{m-2}-1,q^{m-1}-2q^{m-2}-1,q^{m-3}(q-1)^2-1,\frac{q^{m-1}-1}{q-1},q-1). \] 
\end{corollary}

\medskip \noindent We illustrate this with an example, in which a DDG (without loops) arises. Suppose $\zeta$ is a pseudo-polarity in $\PG(2n,q)$ with $n$ a positive integer and $q \geq 3$. The set of all points $x$ of $\PG(2n,q)$ for which $x \in x^\zeta$ is a hyperplane $H^\ast$ of $\PG(2n,q)$ and $x^\ast := (H^\ast)^\zeta$ is a point not contained in $H^\ast$. By Corollaries \ref{co5.6} and \ref{co5.7}, $\Gamma$ and $\widetilde{\Gamma}$ are LDDG's with respective parameters
\[ (q^{2n}-1,q^{2n-1},0,q^{2n-2},\frac{q^{2n}-1}{q-1},q-1), \]
\[ (q^{2n}-1,q^{2n}-q^{2n-1}-1,q^{2n}-2q^{2n-1}-1,q^{2n-2}(q-1)^2-1,\frac{q^{2n}-1}{q-1},q-1). \]
All vertices of $\widetilde{\Gamma}$ have loops, while none of the vertices of $\Gamma$ has loops.

\section*{Acknowledgments}
Anwita Bhowmik and Sergey Goryainov is supported by Natural Science Fundation of Hebei Province (A2023205045).

\end{document}